\documentclass[12pt,a4paper]{article}
\usepackage[english]{babel}
\usepackage[utf8]{inputenc}
\usepackage[T2A]{fontenc}

\usepackage{amsmath}
\usepackage{amssymb}
\usepackage{amsthm}
\usepackage{euscript}
\usepackage[left=2 cm,right=2 cm,top=2.5 cm,bottom=2.5 cm]{geometry}
\usepackage{latexsym}

\usepackage{xcolor}
\usepackage[pdftex, unicode]{hyperref}

\definecolor{linkcolor}{HTML}{00008B}
\definecolor{citecolor}{HTML}{0000CD}
\hypersetup{linkcolor = linkcolor, citecolor = citecolor, colorlinks = true}

\allowdisplaybreaks

\theoremstyle{plain}
  \newtheorem{theorem}{Theorem}
  
  \newtheorem{lemma}{Lemma}
  \newtheorem{propos}{Proposition}
   
\theoremstyle{definition}
  \newtheorem{defin}{Definition}
  \newtheorem{remark}{Remark}
  \newtheorem{problem}{Problem}
  
\numberwithin{equation}{section}

\newenvironment{itemize*}
  {\begin{itemize}
    \setlength{\itemsep}{1pt}
    \setlength{\parskip}{1pt}}
  {\end{itemize}}

\newenvironment{enumerate*}
  {\begin{enumerate}
    \setlength{\itemsep}{1pt}
    \setlength{\parskip}{1pt}}
  {\end{enumerate}}

\title{
Maximization of functionals depending on the terminal value and the running maximum of a martingale: a mass transport approach.}
\author{Nikolay Lysenko\footnote{{
The article was prepared within the framework of the Academic Fund Program at the National Research University Higher School of Economics (HSE) in 2014--2015 (grant \textnumero 14-05-0007) and supported within the framework of a subsidy granted to the HSE by the Government of the Russian Federation for the implementation of the Global Competitiveness Program.
}}}

\begin{document}

\maketitle

\begin{abstract}
It is known that the Az\' ema-Yor solution to the Skorokhod embedding problem maximizes the law of the running maximum of an uniformly integrable martingale with given terminal value distribution. Recently this optimality property has been generalized to expectations of certain bivariate cost functions depending on the terminal value and the running maximum.

In this paper we give an extension of this result to another class of functions. In particular, we study a class of cost functions with the property that the corresponding optimal embeddings are not Az\' ema-Yor. The suggested approach is quite straightforward modulo basic facts of the Monge-Kantorovich mass transportation theory. Loosely speaking, the joint distribution of the running maximum and the terminal value in the Az\' ema-Yor embedding is concentrated on the graph of a monotone function, and we show that this fact follows from the cyclical monotonicity criterion for solutions to the Monge-Kantorovich problem.
\vspace{5pt}

\end{abstract}

\section{Introduction}

Let $W_t$ be the canonical Wiener process and $\mu$ be a centred probability measure on $\mathbb{R}$, i.e. probability measure such that $\int_\mathbb{R} x \: \mu(dx)=0$ and $\int_{\mathbb{R}_{\ge 0}} x \: \mu(dx)< +\infty$. Further $\mu$ will be treated as the law of the terminal value of an 
uniformly integrable martingale started at $0$.

One of the key concepts considered in this paper is the Skorokhod embedding (see \cite{Skor65}). The following formulation can be found in \cite{Hob11}.
\begin{problem}[Skorokhod]
For a given $\mu$ find an uniformly integrable stopping time $\tau$ (with respect to the filtration generated by the canonical Wiener process) such that $W_\tau \sim \mu$.
\end{problem}
Several constructions for $\tau$ are known, and some of them possess different optimality properties. Detailed surveys are presented in works of Ob\l{\'o}j \cite{Obl04} and Hobson \cite{Hob11}. It will be sufficient for our aims to mention only the Az\' ema-Yor embedding~\cite{AY79}, because it plays a significant role in the setup of the paper.

Hobson and Klimmek \cite{HK12} considered maximization of expected values of the cost functions depending on the terminal value and the running maximum. They have shown that the Az\' ema-Yor embedding solves this problem within the class of all uniformly integrable martingales for a broad class of functions. Such functions $F(w, s)$ are assumed to be continuous and differentiable with respect to the running maximum $F_s$. In addition, $F_s / (s - w)$ is assumed to be nondecreasing in $w$. Actually, paper \cite{HK12} contains a lot of other interesting results, but exactly the above fact motivates to study optimality of the Az\' ema-Yor solution for bivariate cost functions.

For additional information about the Az\' ema-Yor embedding one can consult Section~3.1 of \cite{HK12} and the abovementioned comprehensive surveys \cite{Obl04} and \cite{Hob11}. However, the only property of the Az\' ema-Yor embedding used in this paper is as follows. 

\begin{theorem}[Az\' ema and Yor; \cite{AY79}]\label{thm:AzYor}
Consider the class of uniformly integrable martingales with given distribution $\mu$ at the terminal moment $T$. For each positive $l$ the probability that the running maximum of a process from this class is greater than or equal $l$, is not more than $\mu([\beta_\mu^{-1}(l); +\infty))$, where $\beta_\mu^{-1}$ is the inverse barycenter function of $\mu$, i.e. the function inverse to the function $\beta_\mu(k) = \mathbb{E}_\mu [x | x \ge k]$. Moreover, these values are attained for the Az\' ema-Yor embedding  for all $l \ge 0$ simultaneously.
\end{theorem} 

The paper of Hobson and Klimmek uses the results obtained by Rogers \cite{Rog93}. These results are applied in our paper as well. The following necessary and sufficient condition for a measure on $\mathbb{R} \times \mathbb{R}$ to be a joint distribution of the terminal value and the running maximum of an uniformly integrable martingale is of particular importance.

\begin{theorem}[Rogers; based on Corollary 2.4 of \cite{Rog93}]

The measure $\pi$ on $\mathbb{R} \times \mathbb{R}$ is a joint distribution of the final value and the running maximum of an uniformly integrable martingale starting from 0 if and only if all the assumptions (\ref{eq:RogCor1})-(\ref{eq:RogCor4}) are satisfied:

\begin{equation}\label{eq:RogCor1}
\int_{\mathbb{R} \times \mathbb{R}} |x| \: \pi(dx,dy) < +\infty,
\end{equation}

\begin{equation}\label{eq:RogCor2}
\int_{\mathbb{R} \times \mathbb{R}} x \: \pi(dx,dy) = 0,
\end{equation}

\begin{equation}\label{eq:RogCor3}
\mathrm{supp} \, \pi \subset M := (\mathbb{R} \times \mathbb{R}_{\ge 0}) \cap \{(x,y) \, : \, y - x \ge 0 \},
\end{equation}

\begin{equation}\label{eq:RogCor4}
 \mathbb{E}_\pi[ x \, | \, y \ge s] \ge s , \ \ \forall s \ge 0.
\end{equation}

\end{theorem}

Original formulation of Rogers contains one more assumption, namely that $\mathbb{E}_\pi[x \, | \, y \ge s]$ is nondecreasing with respect to $s$. This assumption, however, can be omitted, because it follows immediately from (\ref{eq:RogCor3}) and (\ref{eq:RogCor4}). Indeed, assume existence of $s_1$ and $s_2$ such that $s_1 < s_2$, but $\mathbb{E}_\pi[x \, | \, y \ge s_1] > \mathbb{E}_\pi[x \, | \, y \ge s_2]$.
Note that
$$\mathbb{E}_\pi[x \, | \,  y \ge s_1] = \lambda \mathbb{E}_\pi[x \, | \, s_1 \le y < s_2] + (1-\lambda) \mathbb{E}_\pi[x \, | \, y \ge s_2]$$ 
for some $0 < \lambda \le 1$.
Hence 
$\mathbb{E}_\pi[x \, | \, s_1 \le y < s_2] \ge  \mathbb{E}_\pi[x \, | \, y \ge s_1]  > \mathbb{E}_\pi[x \, | \, y \ge s_2] \ge  s_2$, which is prohibited by (\ref{eq:RogCor3}).

\vspace{5pt}

Application of the Monge-Kantorovich theory might be considered as the third key component of our approach. From rigorous point of view the current problem is not covered by the optimal transport theory, but, fortunately, the reduction to mass transportation problem is possible after some preparations.
The idea to treat the maximization problem for expectation of a bivariate function as a transportation problem with some restrictions appeared in discussions with Alexander Kolesnikov. More information about the optimal transportation theory the interested reader can find in \cite{Vil03} and \cite{BK12}, meanwhile an example of the constrained transportation problem is studied, for instance, in \cite{Zaev14}.

\vspace{5pt}

In order to set the main problem of the article, it is convenient to introduce some definitions designed for internal usage.

\begin{defin}\label{def:Adm}
Given a centred probability measure $\mu$ and a positive $T$, the set of \textit{$\mu$-admissible} (or simply admissible) processes is the set of all uniformly integrable martingales starting from $0$ at $t=0$ which are distributed according to $\mu$ at the terminal moment $t = T$.
\end{defin}

\begin{defin}\label{def:Serr}
A bivariate function $F: \mathbb{R} \times \mathbb{R} \to \mathbb{R}$ is called \textit{serrated} if for every $w_0 \in \mathbb{R}$ the univariate function $F(w_0, s)$ is increasing for $s \! \in \! (-\infty; a_{w_0})$ and decreasing for $s \in (a_{w_0}; +\infty)$, where $a_{w_0} \in \mathbb{R} \cup \{-\infty, +\infty\}$ (value plus or minus infinity means that $F(w_0, s)$ is decreasing or increasing respectively). The set $R_F = \{(w_0, a_{w_0}) : w_0 \in \mathbb{R}\} \subset \mathbb{R} \times (\mathbb{R} \cup \{-\infty, +\infty\})$ is called \textit{ridge} of $F$.
\end{defin}

Recall another useful definition that comes from the applications of the optimal transportation theory.

\begin{defin}\label{def:Supermodular}
A bivariate function $F: \mathbb{R} \times \mathbb{R} \to \mathbb{R}$ is called \textit{supermodular} if it has the following property:
\begin{equation}\label{eq:Supermod}
\forall w_1, w_2, s_1, s_2 \in \mathbb{R} \: \: \left\{ \begin{aligned}
 w_1 < w_2\\
 s_1 < s_2\\
\end{aligned}\right. \: \Rightarrow F(w_1, s_1) + F(w_2, s_2) \ge F(w_1, s_2) + F(w_2, s_1).
\end{equation}
If the rightmost inequality in (\ref{eq:Supermod}) is strict, the function is called \textit{strictly supermodular}.
\end{defin}

\begin{remark}
A strictly supermodular serrated function has an increasing ridge.
\end{remark}

Assume that we are given a continuous strictly supermodular serrated function $F(w, s)$, a centred probability measure $\mu$, and a positive $T$. The problem studied in Section~\ref{Sc:Basic} is as follows: find a $\mu$-admissible process $M_t$ maximizing the expected value of $F(M_T, S_T)$, where $S_T = \max_{0 \le t \le T} M_t$, among of all $\mu$-admissible processes. Despite this case is narrow, it allows to see the core idea clearly. However, the suggested proof is not the simplest one.
 
Actually, it is convenient to assume that $\mu$ has no atoms, because in this case the barycenter function of $\mu$ is continuous and its graph has no vertical intervals. Absence of atoms is not crucial for the reasonings of the paper, but it allows us to get rid of unnecessary complication, so below it is assumed by default.

\vspace{5pt}

If $F(w, s)$ is a strictly supermodular function and $G(w, s)$ is a (not necessary strictly) supermodular function, then their sum \mbox{$H(w, s) = F(w, s) + G(w, s)$} is again strictly supermodular. This obvious fact leads to the idea how a conclusion might be drawn in some cases of discontinuous functions. Such cases are considered in Subsection~\ref{Sc:IntDis}, meanwhile the construction analogous to abovementioned result of \cite{HK12} is represented in Subsection \ref{Sc:TRoHaK}.

\vspace{5pt}

We emphasize that in Section \ref{Sc:Gener} neither differentiability nor even continuity of $F$ is required. Furthermore, examples of proper discontinuous strictly supermodular serrated functions are discussed in Subsection \ref{Sc:IntDis}. Thus, the results obtained in our article are not covered by \cite{HK12}.

\vspace{5pt}

Finally, let us note that the problem under consideration has a natural interpretation in terms of model-independent finance. Suppose that an exotic derivative with payout function $F(M_T, S_T)$ is going to be underwritten at the moment $t = 0$, where $M_t$ is the price process of the underlying asset, and again $S_T = \max_{0 \le t \le T} M_t$. The problem is to determine the no-arbitrage price of the derivative.

What data is available? The idea to retrieve market prognosis from the current quotations of liquid European call options goes back to Breeden and Litzenberger~\cite{BL78}. Their construction allows to reconstruct the measure $\mu$ under assumption of the presence of continuum of liquid calls with the same maturity $T$. Another peace of information is that the underlying asset price process must be an uniformly integrable martingale started at some fixed price. If someone has a model for the underlying asset price, it must be calibrated to this $\mu$ in order to get the exact value. Nevertheless, in model-independent finance all the models that are consistent with available information are treated as realistic, regardless their peculiarities. Thereby the following question arises: what are the upper and the lower bounds for the derivative price? In other words, the seller is interested not in unique price, but in a range of no-arbitrage prices. It is easy to see that this financial problem agrees with the former probabilistic problem.

\vspace{5pt}

The author thanks Alexander Kolesnikov and Alexander Gushchin for their interest and stimulating discussions.

\section{The basic approach: the result for continuous strictly supermodular functions}\label{Sc:Basic}

The lemma below may be considered as an analogue of variational theorems on $c$-cyclical monotonicity of solutions to the Monge-Kantorovich problem. This is the reason why the approach of the paper is called mass transport approach, and a reader not familiar with optimal transportation theory should not be confused by coming to a false conclusion that the mass transport technique means every tricks with measure rearrangement.

\begin{lemma}\label{lem:Variat}
Let $F(w, s)$ be a strictly supermodular continuous function. Assume we are given a centred probability measure $\mu$ and a positive $T$. Then every $\mu$-admissible process maximizing expectation of $F(M_T, S_T)$ has a nondecreasing joint distribution of the terminal value and the running maximum $\pi$. The latter means that there are no intervals \mbox{$I_1, I_2 \subset \mathbb{R}_1$} and $J_1, J_2 \subset \mathbb{R}_2$ (subscripts below $\mathbb{R}$ indicate coordinate axis) with the properties
\begin{equation}\label{eq:VarCond1}
\sup I_1 \le \inf I_2, \: \inf J_1 \ge \sup J_2,
\end{equation}
\begin{equation}\label{eq:VarCond2}
\pi(I_1 \times J_1) > 0, \: \pi(I_2 \times J_2) > 0.
\end{equation}
\end{lemma}

\begin{proof}
Assume that a joint distribution of the terminal value and the running maximum of an admissible process $\pi$ is not nondecreasing. Let us construct a competitor $\overline{\pi}$ on $\mathbb{R} \times \mathbb{R}$, which is a joint distribution generated by an admissible process giving a better value to the cost functional.

Since $\pi$ is not nondecreasing, there are intervals $I_1$, $I_2$, $J_1$ and $J_2$, for which (\ref{eq:VarCond1}) and (\ref{eq:VarCond2}) hold true.

For the sake of technical purposes, it is appropriate to trim them. $I_1 \times J_1$ lies inside a compact set, so there is a point $(x_1, y_1) \in I_1 \times J_1$, such that every $\varepsilon$-neighborhood ($\varepsilon>0$) of this point has a non-zero $\pi$-mass. There exists a point $(x_2, y_2) \in I_2 \times J_2$ with the same property. Consider a compact set $C$ containing both $I_1 \times J_1$ and $I_2 \times J_2$ and then choose $\delta$ little enough to ensure that $\forall (a_1, b_1),(a_2, b_2) \in C$
\begin{equation}\label{eq:UnifCont}
(|(a_1, b_1)-(a_2, b_2)| < \delta) \Rightarrow (|F(a_1, b_1) - F(a_2, b_2)| < Q),
\end{equation}
where $Q = \frac{1}{4}(F(x_1, y_2) + F(x_2, y_1) - F(x_1, y_1) - F(x_2, y_2))$. 
This can be done because of uniform continuity property. Now let $U_1 = (x_1 - \frac{\delta}{2}; x_1 + \frac{\delta}{2})$ and $V_1 = (y_1 - \frac{\delta}{2}; y_1 + \frac{\delta}{2})$. If $y_2 \ne 0$, let $U_2$ and $V_2$ be intervals such that $\mathrm{diam} \, U_2 = \mathrm{diam} \, V_2 < \min(\delta, 2y_2)$ and $(x_2, y_2)$ is the center of the square $U_2 \times V_2$. Else (i.e. if $y_2 = 0$) let $U_2 = (x_2 - \frac{\delta}{2}; x_2 + \frac{\delta}{2})$ and $V_2 = [0; \frac{\delta}{2})$ (it does not matter that $V_2$ is not open). The last conditions are helpful, when proving that (\ref{eq:RogCor3}) remains valid. Remind that the aim is to trim initial intervals, so if $U_1 \times V_1 \subset I_1 \times J_1$ and $U_2 \times V_2 \subset I_2 \times J_2$ are not satisfied, $\delta$ must be chosen little enough to satisfy these conditions.

Now all the preparations are done, and it is possible to construct $\overline{\pi}$ iteratively.

Denote $\min (\pi(U_1 \times V_1), \pi(U_2 \times V_2))$ by $m$. Then construct $\pi_1$ as a measure on $\mathbb{R} \times \mathbb{R}$ that is defined in the following manner:
\begin{enumerate*}
\item it coincides with $\pi$ at subsets of $(\mathbb{R} \times \mathbb{R}) \setminus (U_1 \times V_1)$,
\item for each $\EuScript{S} \subset U_1 \times V_1$ $\pi_1(\EuScript{S}) = \frac{\pi(U_1 \times V_1) - m}{\pi(U_1 \times V_1)}\pi(\EuScript{S})$,
\item it is extended additively to other measurable subsets.
\end{enumerate*}
Note that $\pi_1$ is not probability measure, because its mass equals to $1 - m$. One can treat $\pi_1$ as the result of subtraction of the measure $\sigma_1$ from $\pi$.

Let $\sigma_2$ be a measure of mass $m$ which support is contained inside of $U_2 \times V_1$ and such that for each neighborhood $U$, $U \subset U_2$, $\sigma_2(U \times V_1)$ is proportional to $\pi(U \times V_2)$ and for each neighborhood $V$, $V \subset V_1$, $\sigma_2(U_2 \times V)$ is proportional to $\pi(U_1 \times V)$ (note that $\sigma_2$ is not uniquely defined). Let us define $\pi_2$ as follows:
\begin{enumerate*}
\item $\pi_2$ coincides with $\pi_1$ at subsets of $(\mathbb{R} \times \mathbb{R}) \setminus (U_2 \times V_1)$, \item for each $\EuScript{S} \subset U_2 \times V_1$ $\pi_2(\EuScript{S}) = \pi(\EuScript{S}) + \sigma_2(\EuScript{S})$,
\item it is extended additively to other measurable subsets.
\end{enumerate*}

Here the first step is done. The second (and the final) step is to reassign mass $m$ from $U_2 \times V_2$ to $U_1 \times V_2$ such that the projection on the first coordinate equals to $\mu$. In order to do so define $\pi_3$ in the following manner:
\begin{enumerate*}
\item $\pi_3$ coincides with $\pi_2$ at subsets of $(\mathbb{R} \times \mathbb{R}) \setminus (U_2 \times V_2)$, \item for each $\EuScript{S} \subset U_2 \times V_2$ $\pi_3(\EuScript{S}) = \frac{\pi(U_1 \times V_1) - m}{\pi(U_1 \times V_1)}\pi(\EuScript{S})$,
\item it is extended additively to other measurable subsets.
\end{enumerate*}
The measure $\pi_3$ may be interpreted as the result of subtraction of the measure $\sigma_3$.

Let $\sigma_4$ be a measure of mass $m$ supported inside of $U_1 \times V_2$ such that for each neighborhood $U$, $U \subset U_1$, $\sigma_4(U \times V_2)$ is proportional to $\pi(U \times V_1)$ and for each neighborhood $V$, $V \subset V_2$, $\sigma_4(U_1 \times V)$ is proportional to $\pi(U_2 \times V)$ (again $\sigma_4$ is not uniquely defined). Finally, define $\pi_4$ as a measure with the properties
\begin{enumerate*}
\item $\pi_4$ coincides with $\pi_3$ on subsets of $(\mathbb{R} \times \mathbb{R}) \setminus (U_1 \times V_2)$, \item for each $\EuScript{S} \subset U_1 \times V_2$ $\pi_4(\EuScript{S}) = \pi(\EuScript{S}) + \sigma_4(\EuScript{S})$,
\item $\pi_4$ is extended additively to other measurable subsets
\end{enumerate*}

It remains to prove that $\overline{\pi} = \pi_4$ is the desired competitor.

First we observe that $\mathrm{Pr}_1 \, \pi_4 = \mu$, because $\mathrm{Pr}_1 \, \sigma_1 = \mathrm{Pr}_1 \, \sigma_4$ and $\mathrm{Pr}_1 \, \sigma_2 = \mathrm{Pr}_1 \, \sigma_3$. This fact implies that  (\ref{eq:RogCor1}) and (\ref{eq:RogCor2}) are satisfied.

Assumption (\ref{eq:RogCor3}) is satisfied too, because $(x_2, y_2) \in M = (\mathbb{R} \times \mathbb{R}_{\ge 0}) \cap \{(x,y) \, : \, y - x \ge 0 \}$, and this provides that $U_1 \times V_2$ and $U_2 \times V_1$ are subsets of the required set $M$; it can be clearly seen from the construction.

Further, (\ref{eq:RogCor4}) holds true, because $\forall s \in \mathbb{R}_{\ge 0}$ the three statements are true: 
\begin{equation}\label{eq:ProofRogCor4-1}
\mathbb{P}_\pi[ x \in U_1 \, | \, y \ge s] \ge \mathbb{P}_{\pi_4}[ x \in U_1 \, | \, y \ge s],
\end{equation}
\begin{equation}\label{eq:ProofRogCor4-2}
\mathbb{P}_\pi[ x \in U_2 \, | \, y \ge s] \le \mathbb{P}_{\pi_4}[ x \in U_2 \, | \, y \ge s],
\end{equation}
\begin{equation}\label{eq:ProofRogCor4-3}
\mathbb{P}_\pi[ x \notin U_1 \, \mathrm{and} \, x \notin U_2 \, | \, y \ge s] = \mathbb{P}_{\pi_4}[ x \notin U_1 \, \mathrm{and} \, x \notin U_2 \, | \, y \ge s].
\end{equation}
Putting (\ref{eq:ProofRogCor4-1})-(\ref{eq:ProofRogCor4-3}) together yields the desired result that $\forall s \in \mathbb{R}_{\ge 0}$
\begin{equation}\label{eq:ProofRogCor4-T}
\mathbb{E}_{\pi_4}[ x \, | \, y \ge s] \ge \mathbb{E}_\pi[ x \, | \, y \ge s] \ge s.
\end{equation}

Thus, $\pi_4$ is the joint distribution of the terminal value and the running maximum of an uniformly integrable martingale starting from 0 and, moreover, this process is $\mu$-admissible, because the law of the terminal value is $\mu$.

Finally, we need to prove that the process corresponding to $\pi_4$ is better. To this end we prove the following inequality:
\begin{equation}\label{eq:RecInForAreas}
\int F(w, s) \, (\sigma_2 + \sigma_4)(dw, ds) - \int F(w, s) \, (\sigma_1 + \sigma_3)(dw, ds) > 0.
\end{equation}
The desired result follows from the following line of computations:

\begin{align*}
\int & F(w, s)  (\sigma_2 + \sigma_4)(dw, ds)   - \int F(w, s) (\sigma_1 + \sigma_3)(dw, ds)
\\& =
\int \bigl(F(x_1, y_2) + (F(w, s) - F(x_1, y_2)\bigr)  \sigma_2(dw, ds)
\\& +
\int \bigl(F(x_2, y_1) + (F(w, s) - F(x_2, y_1)\bigr)  \sigma_4(dw, ds) 
\\& -
\int \bigl(F(x_2, y_2) + (F(w, s) - F(x_2, y_2)\bigr)  \sigma_3(dw, ds)
\\& -
\int \bigl(F(x_1, y_1) + (F(w, s) - F(x_1, y_1)\bigr)  \sigma_1(dw, ds)
\\& =
\int F(x_1, y_2)  \sigma_2(dw, ds) + \int (F(w, s) - F(x_1, y_2))  \sigma_2(dw, ds) 
\\& +
\int F(x_2, y_1)  \sigma_4(dw, ds) + \int (F(w, s) - F(x_2, y_1))  \sigma_4(dw, ds) 
\\& -
\int F(x_2, y_2)  \sigma_3(dw, ds) - \int (F(w, s) - F(x_2, y_2))  \sigma_3(dw, ds)
\\&  -
\int F(x_1, y_1)  \sigma_1(dw, ds) - \int (F(w, s) - F(x_1, y_1))  \sigma_1(dw, ds)
\\& =
(F(x_1, y_2) + F(x_2, y_1) - F(x_2, y_2) - F(x_1, y_1))m 
\\& +
\int \bigl(F(w, s) - F(x_1, y_2)\bigr)  \sigma_2(dw, ds) +
\int \bigl(F(w, s) - F(x_2, y_1)\bigr)  \sigma_4(dw, ds) 
\\& -
\int \bigl(F(w, s) - F(x_2, y_2)\bigr)  \sigma_3(dw, ds) -
\int \bigl(F(w, s) - F(x_1, y_1)\bigr)  \sigma_1(dw, ds) 
\\& >
\bigl(F(x_1, y_2) + F(x_2, y_1) - F(x_2, y_2) - F(x_1, y_1)\bigr) m 
\\& -
\Bigl(\int Q  \sigma_2(dw, ds) +
 \int Q  \sigma_3(dw, ds) +
 \int Q  \sigma_4(dw, ds) +
 \int Q  \sigma_1(dw, ds)\Bigr) = 0.
\end{align*}

This completes the proof.
\end{proof}

With the help of Lemma \ref{lem:Variat} we establish several results on optimal admissible processes. We start with two extremal types of strictly supermodular serrated functions and get a generalizing statement at the end of this section.

\begin{theorem}\label{thm:AzYorOpt}
Let $F(w, s)$ be a continuous strictly supermodular serrated function, and let its ridge $R_F$ be the set $\{(w, +\infty) \, : \, w \in \mathbb{R}\}$. Then the functional $\mathbb{E}[F(w, s)]$ considered on the set of all $\mu$-admissible processes is maximized by the Az\' ema-Yor embedding.
\end{theorem}

\begin{proof}
By Lemma \ref{lem:Variat}, the joint distribution generated by an optimal process must have a nondecreasing support. A $\mu$-admissible process can not generate a distribution with nondecreasing support which has a non-zero mass higher than the graph of the barycenter function $\beta_\mu$. It follows from the Az\' ema-Yor embedding optimality property for nondecreasing functions depending only on the running maximum (see Theorem \ref{thm:AzYor}). Thereby all the mass is placed not higher than the barycenter function graph. Since for every $w_0$ the univariate functions $F(w_0, s)$ is assumed to be increasing, the Az\' ema-Yor embedding is the optimal admissible process.

It can be seen from the following line of computations, where $U$ stands for the joint distribution of an arbitrary $\mu$-admissible process, $U_w$ stands for the conditional distribution given fixed $w$, and $AY$ stands for the joint distribution of the Az\' ema-Yor embedding:
\begin{equation*}
\mathbb{E}_U [F(M_T, S_T)] = \int \Bigl( \int F(w, s) U_w(ds) \Bigr) \mu(dw) \le \int F(w, \beta_\mu(w)) \mu(dw) = \mathbb{E}_{AY} [F(M_T, S_T)].
\end{equation*}
The inequality is equality if $U$ is equivalent to $AY$.
\end{proof}

\begin{remark}
An example of a function satisfying assumptions of Theorem \ref{thm:AzYorOpt} is $F(w, s) = (\arctan w + 2)s$.
\end{remark}

\begin{propos}\label{prop:PureJumpOpt}
Let $F(w, s)$ be a measurable serrated function, and let its ridge $R_F$ be the set $\{(w, -\infty) \, : \, w \in \mathbb{R}\}$. Then the functional $\mathbb{E}[F(w, s)]$ considered on the set of all $\mu$-admissible processes is maximized by the the pure jump process, i.e. the process which is constant at the time interval $[0; T/2)$, jumps to a value of a random variable with law $\mu$ at the $t = T/2$, and equals to another constant at the time interval $[T/2; T]$.
\end{propos}

\begin{proof}
The described process generates the joint distribution with support contained in the boundary of $M$ (see (\ref{eq:RogCor3}) for the definition of $M$). Since (\ref{eq:RogCor3}) prohibits placing a non-zero mass lower than this boundary and since for every $w_0$ the univariate functions $F(w_0, s)$ is assumed to be decreasing, this is an optimal admissible process.

Again, it can be seen from the following line of computations, where $U$ stands for the joint distribution of an arbitrary $\mu$-admissible process, $U_w$ stands for the conditional distribution given fixed $w$, and $PJ$ stands for the joint distribution of the pure jump process:
\begin{equation*}
\mathbb{E}_U [F(M_T, S_T)] = \int \Bigl( \int F(w, s) U_w(ds) \Bigr) \mu(dw) \le \int F(w, \max(0, w)) \mu(dw) = \mathbb{E}_{PJ} [F(M_T, S_T)].
\end{equation*}
The inequality is equality if $U$ is equivalent to $PJ$. 
\end{proof}

\begin{remark}
An example of a function satisfying assumptions of Proposition \ref{prop:PureJumpOpt} is $F(w, s) = -|w|s$.
\end{remark}

\begin{remark}
In Proposition \ref{prop:PureJumpOpt} $\mu$-admissibility can be replaced by less restrictive condition, since martingale property is not used. However, the optimal process (to be precise, at least one of the optimal processes) is still martingale.
\end{remark}

We will see, that, roughly speaking, a general optimal process has the joint distribution of the final value and the running maximum in an ''intermediate position'' between the two discussed extremal distributions.

\begin{lemma}\label{lem:MonIsJoint}
Every nondecreasing function $f$ taking values in the intersection of $M$ with the closed subgraph of the barycenter function of $\mu$, induces a measure that is the joint distribution of the terminal value and the running maximum of a $\mu$-admissible process. 
\end{lemma}

\begin{proof}
Consider the measure $f_\# \mu$. The properties (\ref{eq:RogCor1}) and (\ref{eq:RogCor2}) are immediate. 
The formulation of the Lemma states that (\ref{eq:RogCor3}) is satisfied. Finally, (\ref{eq:RogCor4})
follows from the inequality
\begin{equation}
\mathbb{E}_{f_\# \mu} [x | y \ge s] \ge \mathbb{E}_{AY} [x | y \ge s] \ge s,
\end{equation}
where subscript AY refers to the distribution generated by the Az\' ema-Yor embedding (note that $\mathbb{E}_{AY} [x | y \ge s] 
=s$ in any standard situation).
\end{proof}

\begin{theorem}\label{thm:General}
Let $F(w, s)$ be a continuous strictly supermodular serrated function. Then amongst all of $\mu$-admissible processes the functional $\mathbb{E}[F(w, s)]$ is maximized by the process with the property that the corresponding joint distribution of the final value and the running maximum is $g_\# \mu$, where $g$ is defined by
\begin{equation}
g(w) = \min(\beta_\mu(w), \max(a_w, 0, w)),
\end{equation}
$\beta_\mu$ is the same as in Theorem~\ref{thm:AzYor}, and $a_w$ is the same as in Definition \ref{def:Serr}.
\end{theorem}

\begin{proof}
The described process exists and belongs to $\mu$-admissible processes, because the maximum of two monotone functions is a monotone function and the minimum of two monotone functions is again a monotone function, so $g$ is covered by Lemma~\ref{lem:MonIsJoint}. Explicit description of this process can be found in the proof of Theorem 2.2 of \cite{Rog93}.

The process is optimal, because of the following reasoning. Lemma~\ref{lem:Variat} implies that the optimal joint distribution of the terminal value and the running maximum of an admissible process must have monotone support. A joint distribution with monotone support can not place non-zero mass beyond the closed subgraph of the barycenter function of $\mu$, since it is prohibited by Theorem~\ref{thm:AzYor}. Also condition (\ref{eq:RogCor3}) states that every joint distribution of the terminal value and the running maximum of an admissible process can not place non-zero mass beyond $M$. Combining this together, obtain that for each fixed $w$ within the region where all mass must be placed the best possible point is the point that belongs to the graph of $g$. Thus, the described in the formulation process is optimal.

As before, the last statement can be seen from the following line of computations, where $U$ stands for the joint distribution of an arbitrary $\mu$-admissible process and $U_w$ stands for the conditional distribution given fixed $w$:
\begin{equation*}
\mathbb{E}_U [F(M_T, S_T)] = \int \Bigl( \int F(w, s) U_w(ds) \Bigr) \mu(dw) \le \int F(w, g(w)) \mu(dw) = \mathbb{E}_{g_\# \mu} [F(M_T, S_T)].
\end{equation*}
The inequality is equality if $U$ is equivalent to $g_\# \mu$. 
\end{proof}

\begin{remark}
The following function satisfies assumptions of Theorem \ref{thm:General}: 
$$F(w, s) = (\arctan w + 2)s - 4 \, \mathrm{Ind}_{\{(w, s) \, : \, s \ge R(w)\}}(s - R(w)),$$ where $R$ is an increasing function. The ridge of $F(w, s)$ is the graph of $s = R(w)$. 
\end{remark}

\begin{remark}
In the above theorems when it is talked about maximization of the functional $\mathbb{E}[F(M_T, S_T)]$, it means that there are no processes that allow achieving higher values. If a \mbox{$\mu$-admissible} process $U$ is such that one of the theorems states that it is optimal for corresponding $F$, but $\mathbb{E}_U[F(M_T, S_T)] = -\infty$, then, of course, all other $\mu$-admissible processes are also optimal for the same problem.
\end{remark}

\section{Generalizations}\label{Sc:Gener}

\subsection{Introducing discontinuity}\label{Sc:IntDis}

We mentioned already that the above approach to the proofs of Theorems \ref{thm:AzYorOpt} and \ref{thm:General} is not the simplest one. Instead of it one can apply directly
the optimal transportation theory. 
The standard assumption assuring existence of the solution to the Monge-Kantorovich problem is the lower semicontinuity of the cost function. In addition, the solutions to the Monge-Kantorovich problem admit the so-called cyclical monotonicity property, which can be established under assumptions that at least are not stricter than assumption of lower semicontinuity (see \cite{BK12} for references to recent results in this direction).

\vspace{5pt}

To start with, let us give some basic definitions.

\begin{problem}
Suppose that $c: \mathbb{R} \times \mathbb{R} \to \mathbb{R} \cup \{\infty\}$ is a measurable function (often it is called cost function). Suppose also that $\mu$ and $\nu$ are Borel measures on $\mathbb{R}$ and $\Pi(\mu, \nu)$ is the set of all Borel measures $\pi$ on $\mathbb{R} \times \mathbb{R}$ such that $\mathrm{Pr}_1 \, \pi = \mu$ and $\mathrm{Pr}_2 \, \pi = \nu$. \textit{The Monge-Kantorovich problem} is:
\begin{equation}
\int_{\mathbb{R} \times \mathbb{R}} c(x, y) \pi(dx, dy) \to \min\limits_{\pi \in \Pi(\mu, \nu)}.
\end{equation}
\end{problem}

Actually, above definition is not general. It is possible to consider $X \times Y$ for measure spaces $X$ and $Y$ instead of $\mathbb{R} \times \mathbb{R}$ or to introduce more than two axes (also known as marginals).

\begin{defin}
Suppose that $c: \mathbb{R} \times \mathbb{R} \to \mathbb{R} \cup \{\infty\}$ is a measurable function. The subset $\Gamma \subset \mathbb{R} \times \mathbb{R}$ is called $c$\textit{-cyclically monotone} if for every non-empty sequence of its elements $(x_1, y_1)$, ..., $(x_n, y_n)$ it is true that:
\begin{equation}
c(x_1, y_1) + c(x_2, y_2) + ... + c(x_n, y_n) \le c(x_1, y_n) + c(x_2, y_1) + ... + c(x_n, y_{n-1}).
\end{equation}
\end{defin}

\vspace{5pt}

Below the previously declared in Introduction approach is represented.

\begin{theorem}\label{thm:LSC}
In the formulations of Theorems \ref{thm:AzYorOpt} and \ref{thm:General} continuity can be replaced by upper semicontinuity.
\end{theorem}

\begin{proof}
Suppose that a joint distribution $\pi$ of the final value and the running maximum of a $\mu$-admissible process is given. Denote $\mathrm{Pr}_2 \, \pi$ as $\nu$. Let $\overline{\pi}$ be a solution to the Monge-Kantorovich problem with the marginals $\mu$ and $\nu$ and the cost function $-F(w,s)$. Here the sign is reversed, because the initial problem is a maximization problem, but the Monge-Kantorovich problem is a minimization problem. Solution to the described Monge-Kantorovich problem exists, because the cost function $-F$ is lower semicontinuous.

It can be easily verified that the $c$-cyclical monotonicity (i.e. $-F(w,s)$-cyclical monotonicity) implies that $\overline{\pi}$ is concentrated on the graph of a monotone function $T$ (this is a standard observation coming from the optimal transportation theory):
\begin{equation}\label{eq:GraphConc}
\overline{\pi} \Bigl(\{(x, T(x)): \ x \in \mathbb{R} \}\Bigr)=1.
\end{equation}

Let us show that $\overline{\pi}$ is the joint distribution of an admissible process. Conditions (\ref{eq:RogCor1}) and (\ref{eq:RogCor2}) are satisfied automatically. Further, $\mathrm{supp} \, \overline{\pi} \subset \mathbb{R} \times \mathbb{R}_{\ge 0}$, because $\mathrm{supp} \, \nu \subset \mathbb{R}_{\ge 0}$, meanwhile $\mathrm{supp} \, \overline{\pi} \subset \{(x,y) \, : \, y - x \ge 0 \}$ due to the following reasoning. For the initial joint distribution $\pi$ it is true that:
\begin{equation}\label{eq:MeasLoc}
\forall k > 0 \; \nu([0;k]) = \pi(\mathbb{R} \times [0;k]) = \pi((-\infty; k] \times [0;k]) \le \mu((-\infty; k]).
\end{equation}
Combining (\ref{eq:MeasLoc}) with (\ref{eq:GraphConc}) yields that (\ref{eq:RogCor3}) is checked. Finally, (\ref{eq:RogCor4}) holds true, because $\forall k > 0$ $\nu([k; +\infty)) \le \nu_{AY}([k; +\infty))$, where subscript AY refers to the joint distribution generated by the Az\' ema-Yor embedding, so $\mathbb{E}_{\overline{\pi}} [x | y \ge s] \ge \mathbb{E}_{AY} [x | y \ge s] \ge s$.

The construction ensures that the process that generates $\overline{\pi}$ is not worse than the initial process.
 Since its joint distribution is supported on the graph of a monotone function, 
the further proof can follow the arguments of Theorems \ref{thm:AzYorOpt} and \ref{thm:General}.
\end{proof}

However, as far as the author knows, in the Monge-Kantorovich theory there is no analogous result for upper semicontinuity instead of lower semicontinuity. Fortunately, the approach of Section \ref{Sc:Basic} is applicable to some lower semicontinuous functions.

\begin{propos}
Assume that $F(w,s)$ is a continuous supermodular serrated function and $G_{a,b}(w,s) = \mathrm{Ind}_{\{(w,s) \: : \: w \ge a, \, s \ge b\}}$, where $(a, b)$ is in the ridge of $F(w,s)$. Then Theorems \ref{thm:AzYorOpt} and \ref{thm:General} hold true for $(F+G_{a,b})(w,s)$, which is lower semicontinuous supermodular serrated function with the same ridge.
\end{propos}

\begin{proof}
$(F+G_{a,b})(w,s)$ possesses declared properties due to its construction. It is sufficient to prove only generalization of Lemma \ref{lem:Variat}, because other reasonings of the theorems are still applicable.

Consider four measures $\sigma_1$, $\sigma_2$, $\sigma_3$ and $\sigma_4$ which are applied in the proof of Lemma~\ref{lem:Variat} and have equal masses $m$. The desired result follows from the inequality:
\begin{equation}\label{eq:Indicator}
\int G_{a,b}(w,s) (\sigma_2 + \sigma_4 - \sigma_1 - \sigma_3)(dw, ds) \ge 0.
\end{equation}
The above inequality can be proved by trivial analysis of possible configurations of supports of this four measures relatively $\{(w,s) \: : \: w \ge a, \, s \ge b\}$. This analysis should be based on applications of the equalities $\mathrm{Pr}_1 \, \sigma_1 = \mathrm{Pr}_1 \, \sigma_4$, $\mathrm{Pr}_2 \, \sigma_1 = \mathrm{Pr}_2 \, \sigma_4$, $\mathrm{Pr}_1 \, \sigma_2 = \mathrm{Pr}_1 \, \sigma_3$, and $\mathrm{Pr}_2 \, \sigma_3 = \mathrm{Pr}_2 \, \sigma_4$. It is just plain geometry and arithmetic.
\end{proof}

\begin{remark}
Of course, the previous proposition is true not only for $\mathrm{Ind}_{\{(w,s) \: : \: w \ge a, \, s \ge b\}}$, but it also remains valid for $\mathrm{Ind}_{\{(w,s) \: : \: w \le a, \, s \le b\}}$, $-\mathrm{Ind}_{\{(w,s) \: : \: w > a, \, s < b\}}$, $-\mathrm{Ind}_{\{(w,s) \: : \: w < a, \, s > b\}}$, and positive linear combinations of such functions. 
\end{remark}

\subsection{Towards the result of Hobson and Klimmek}\label{Sc:TRoHaK}

Everywhere above the second marginal of the joint distribution was fixed. This is consistent with the spirit of the Monge-Kantorovich theory, but the price for this is a quite restrictive requirement of supermodularity. If $F(w,s)$ is continuously differentiable with respect to $s$, its supermodularity is equivalent to the assumption that all the univariate functions $F_s(w,s_0)$ with arbitrarily fixed $s_0$, are increasing. In \cite{HK12} a weaker assumption was suggested: why not suppose that only $F_s(w,s_0) / (s_0 - w)$ is monotonic? Actually, attentive reader can see that in the proof of Lemma~\ref{lem:Variat} in (\ref{eq:ProofRogCor4-T}) condition (\ref{eq:RogCor4}) is satisfied in a non-optimal manner and for some $s$ inequality $\mathbb{E}_{\overline{\pi}}[ x \, | \, y \ge s] \ge s$ is strict, which means that an additional mass may be lifted up. This is refined in the current subsection.

To make the idea clear, do the following. Consider the function $F(w,s)$ with continuous partial derivative $F_s(w,s)$. Moreover, we assume that for each $s_0 > 0$ the function $F_s(w,s_0) / (s_0 - w)$ is increasing for $w < s_0$. Suppose that $\pi$ is a joint distribution of a martingale with terminal value law $\mu$ and that there are points $(w_1, s_1)$ and $(w_2, s_2)$, with $w_1 < w_2$ and $s_1 > s_2$, such that $\pi$ has atoms with masses at least $m$ at this points. Remove mass $(s_2 - w_2)/(s_2 - w_1)m$ from $(w_1,s_1)$ to $(w_1,s_2)$ and remove mass $m$ from $(w_2,s_2)$ to the point $(w_2, s_1)$ and the vertical open interval between $(w_2,s_2)$ and $(w_2,s_1)$ in such a way that for every $q$, $s_2 < q \le s_1$, the mass lifted not lower than $q$ is $((q-w_1)(s_2 - w_2))/((q-w_2)(s_2 - w_1))m$. It can be seen clearly that the coefficients were chosen in order to provide that (\ref{eq:RogCor4}) remains true. Because the properties (\ref{eq:RogCor1})-(\ref{eq:RogCor3}) here are trivial, new measure is a joint distribution as well. Moreover, this measure is a better competitor, since the gain from the reassignment is:
\begin{equation}
\Delta := \int_{(w_2,s_2)}^{(w_2,s_1)}F_s(w_2, \xi)\frac{(\xi-w_1)(s_2 - w_2)}{(\xi-w_2)(s_2 - w_1)}m \, d\xi - \int_{(w_1,s_2)}^{(w_1,s_1)}F_s(w_1, \xi)\frac{(s_2 - w_2)}{(s_2 - w_1)}m \, d\xi.
\end{equation}
The assumptions on $F(w,s)$ imply that the latter is positive.

The proof of the following statement is omitted because up to some inessential technicalities it is reduced to the proof of an appropriate analogue of Lemma \ref{lem:Variat} and follows the same line. 

\begin{theorem}
The results of Theorems \ref{thm:AzYorOpt} and \ref{thm:General} hold true under the following assumptions imposed on $F$ instead of strict supermodularity: $F(w,s)$ is continuous, there exists a continuous partial derivative $F_s(w,s)$, and, in addition, for each $s_0 > 0$ the function $F_s(w,s_0) / (s_0 - w)$ is increasing for $w < s_0$.
\end{theorem}

\begin{remark}
The case of nondecreasing $F_s(w, s_0) / (s_0 - w)$ for every $s_0 > 0$ and $w < s_0$ is considered in \cite{HK12}, but it is not covered here. 
\end{remark}


\vspace{20pt}

\begin{thebibliography}{100}
\addcontentsline{toc}{section}{References}

\bibitem{AY79} Az\' ema,~J. and Yor,~M.; Le probl\' eme de Skorokhod: Compl\' ements \'a "Une solution simple au probl\' eme de Skorokhod". In S\' eminaire de Probabilit\' es, XIII (Univ. Strasbourg, Strasbourg), 1979.

\bibitem{BK12} Bogachev,~V.I. and Kolesnikov,~A.V.; The Monge--Kantorovich problem: achievements, connections, and perspectives. Uspekhi Mat. Nauk, 67, 5(407), 3--110, 2012. 

\bibitem{BL78} Breeden,~D.T. and Litzenberger,~R.H.; Prices of state-contingent claims implicit in options prices. J. Business, 51, 621--651, 1978.

\bibitem{Hob11} Hobson,~D.G.; The Skorokhod embedding problem and model-independent bounds for option prices. Volume 2003 of Lecture Notes in Math., 267--318, Springer, Berlin, 2011.

\bibitem{HK12} Hobson,~D.G. and Klimmek,~M.; Maximizing functionals of the maximum in the Skorokhod embedding problem and an application to variance swaps. The Annals of Applied Probability, Vol. 23, No. 5, 2020--2052, 2013.

\bibitem{Obl04} Ob\l{\'o}j,~J.; The Skorokhod embedding problem and its offspring. Probab. Surv., 1, 321--390 (electronic), 2004.

\bibitem{Rog93} Rogers,~L.C.G.; The joint law of the maximum and terminal value of a martingale. Probab. Theory Related Fields, 95, 451--466, 1993.

\bibitem{Skor65} Skorokhod,~A.V.; Studies in the Theory of Random Processes. Addison-Wesley, Reading, MA, 1965.

\bibitem{Vil03} Villani,~C.; Topics in optimal transportation. Grad. Stud. Math., 58, Amer. Math. Soc., Providence, RI, 2003, xvi + 370 pp.

\bibitem{Zaev14} Zaev,~D.A.; On the Monge-Kantorovich problem with additional linear constraints. Mat. Zametki, 98(5), 664--683, 2015.

\end{thebibliography}
\end{document}